\documentclass[12pt,a4paper]{amsart}
\usepackage[latin1]{inputenc}
\usepackage[english]{babel}
\usepackage{amsmath,amssymb,amsthm}
\usepackage{color}

\usepackage{enumerate}
\usepackage{cite}
\usepackage{mathtools}

 \theoremstyle{plain}
\newtheorem{theorem}{Theorem}[section]
\newtheorem{lemma}[theorem]{Lemma}
\newtheorem{corollary}[theorem]{Corollary}
\newtheorem{prop}[theorem]{Proposition}

\newtheorem{proposition}[theorem]{Proposition}

\theoremstyle{definition}

\newtheorem{example}[theorem]{Example}

\def\T{\mathbb{T}}
\def\Z{\mathbb{Z}}
\def\C{\mathbb{C}}

\def\R{\mathbb{R}}
\let\phi\varphi
\let\epsilon\varepsilon

\begin{document}
\title[on isomorphisms between weighted $L^p$-algebras]{on isomorphisms between weighted $L^p$-algebras}

\author[Y. Kuznetsova]{Yulia Kuznetsova$^1$}

\address{Laboratoire de Math\'ematiques de Besan\c{c}on,
Universite Bourgogne Franche-Comt\'e, France}
\email{\texttt{yulia.kuznetsova@univ-fcomte.fr}}

\author[S. Zadeh]{Safoura Zadeh$^2$}
\address{Laboratoire de Math\'ematiques de Besan\c{c}on,
Universite Bourgogne Franche-Comt\'e, France \& Faculty of Graduate Studies, Dalhousie University, Canada.
}
\email{\texttt{jsafoora@gmail.com}}


\footnotetext[1]{This work was supported by the French ``Investissements d'Avenir'' program, project ISITE-BFC (contract ANR-15-IDEX-03).}
\footnotetext[2]{Research supported by the incoming mobility program of the Region of Bourgogne-Franche-Comt\'e, France}
\begin{abstract}
Let $G$ be a locally compact group and $\omega$ be a continuous weight on $G$. In this paper, we first characterize bicontinuous biseparating algebra isomorphisms between weighted $L^p$-algebras. As a result we extend previous results of Edwards, Strichartz and Parrot on algebra isomorphisms between $L^p$-algebras to the setting of weighted $L^p$-algebras. We then study the automorphisms of certain weighted $L^p$-algebras on integers, applying known results on composition operators on classical function spaces.
\end{abstract}
\maketitle
\section{Introduction and Preliminaries}\label{sec1}
Let $G$ be a locally compact group with a fixed Haar measure denoted by $\lambda$. By the group algebra on $G$ we mean the convolution Banach algebra $L^1(G)$ of Haar-integrable functions on $G$. It is well-known that continuous unitary representations of $G$ correspond canonically to non-degenerate involutive representations of $L^1(G)$. This makes one wonder how far this correspondence goes, and in particular, if the existence of an algebra isomorphism between two group algebras implies that the underlying groups are isomorphic. In \cite{MR0049910}, Wendel showed that, in general, the answer to this question is negative. Therefore, we have to impose some constraints on the isomorphisms. 

In \cite{MR0025472}, Kawada proved that if we have a bipositive algebra isomorphism between group algebras, then the underlying locally compact groups must be isomorphic. Wendel \cite{MR0049910} showed that the same conclusion holds for isometric isomorphisms between group algebras.
If  we assume the algebra isomorphism is of small bound instead of isometric, then as it is proved by Kalton and Wood \cite{MR0415205}, the existence of an algebra isomorphism of norm less than a constant $\gamma\simeq1.2\dots$ between the group algebras implies that the underlying topological groups must be isomorphic. Wood obtained later many interesting results on similar questions (see for example \cite{MR0463808}, \cite{MR1089950}, \cite{MR1358183} and \cite{MR1871399}).

The same questions have been asked for $L^p$-convolution algebras of compact groups with $p>1$. In \cite{MR0183807}, Edwards proved that for $1\leq p<\infty$, the existence of a bipositive algebra isomorphisms between $L^p$-algebras of compact groups implies that the groups are isomorphic, and asked if a similar result holds if we consider an isometric algebra isomorphism instead. Parrot \cite{MR0227323} and Strichartz \cite{MR01935310}, independently, provided an affirmative answer to Edward's question for the case of compact groups with $1\leq p<\infty,\ p\neq2$. When $p=2$, as it is also mentioned in \cite[Page 861]{MR01935310}, $L^2(G)$ does not determine $G$. In fact, if for every positive integer $k$, the cardinality of the set of equivalence classes of irreducible representations of dimension $k$ of $G$ is equal to the cardinality of the corresponding set of $H$, then $L^2(G)$ is isometrically isomorphic to $L^2(H)$.  

Adding a weight can considerably enlarge the class of algebras associated with groups. Let $G$ be a locally compact group, and let $\omega:G\to\mathbb{R}^+$ be a positive continuous function. Define the Banach space $L^p(G,\omega)$ as
$$
L^p(G,\omega):=\{f:G\to\mathbb{C} \ |\ f\text{ is measurable and } \int_G\lvert f(x)\rvert^p\omega(x)^p dx<\infty\},
$$
with the norm $\|f\|_{p,\omega}:=\left(\int_G|f(x)|^p\omega(x)^p dx\right)^{1/p}$. So, when the weight $\omega\equiv 1$, we have the classical definition of $L^p$-spaces.

If $\omega (xy)\leq\omega(x)\omega(y)$ for all $x$, $y\in G$, then $L^1(G,\omega)$ is an algebra with respect to convolution. Depending on the weight, these algebras can be quite different from $L^1(G)$: An example is given already by the Beurling's result \cite{beurling} on whether or not $L^1(G,\omega)$ possesses the regularity property.

If $G$ is not compact and $p\neq1$, $L^p(G)$ is not an algebra for convolution. However, weighted $L^p$-spaces on non-compact groups, and in particular on any sigma-compact group \cite{kuz-mz} may form Banach algebras. For $p>1$, a sufficient condition for $L^p(G,\omega)$ to be a convolution algebra is $\omega^{-q}*\omega^{-q}\le\omega^{-q}$ with $q$ such that $1/p+1/q=1$. Virtually all weights occuring in practice satisfy this condition. 

It is natural to ask whether the presence of a weight can enlarge the set of isomorphisms between group algebras. Surprisingly, a weighted structure is as rigid as an unweighted one: if there exists a bipositive \cite{MR3589852} or isometric \cite{MR3462562} algebra isomorphism between $L^1(G,\omega_1)$ and $L^1(H,\omega_2)$ then the groups $G,H$ are isomorphic, and the weights are linked by a multiplicative relation.

In Section 2 of this paper we investigate isomorphisms of weighted $L^p$-algebras on locally compact groups with $p>1$. 
Our approach unites both problems of characterizing bipositive algebra isomorphism and isometric algebra isomorphism into one problem of characterizing bicontinuous biseparating algebra isomorphisms (an operator $T$ is biseparating if both $T$ and $T^{-1}$ are disjointness preserving). By doing so, we conclude that the existence of a bipositive or isometric ($p\ne2$) algebra isomorphism between $L^p(G,\omega_1)$ and $L^p(H,\omega_2)$, $p\geq1$, implies that the groups $G,H$ are isomorphic, and the isomorphism is of a canonical form: a weighted composition operator induced by a group character and a group isomorphism.

As indicated before, without restriction on the norm there are isomorphisms of group algebras which are not expresssed in this canonical form, even if the two groups are the same. This can be observed most easily on the examples of finite groups. There is however a series of results which show that for certain groups, first and foremost $\Z$, there are very few isomorphisms of groups algebras. It follows from the Beurling-Helson theorem that every isomorphism of $\ell^1(\Z)$ is either identical or the composition with the map $\phi: n\mapsto -n$.

In Section 3 we apply some variants of the Beurling-Helson theorem to study the automorphisms of the classical algebras $\ell^p(\Z,\omega)$ without restriction on the norm. The result is that for polynomial weights, there is no nontrivial automorphisms for all $p\ge1$, $p\ne2$. We also provide some examples with fast-growing weights.

\bigskip

{\bf Notation.} In the sequel, a weight on $G$ is a positive continuous function $\omega:G\to\mathbb{R}^+$ such that $\omega (xy)\leq\omega(x)\omega(y)$, for all $x$, $y\in G$ and $\omega(e_G)=1$, where $e_G$ is the identity element of $G$. By a weighted locally compact group, we mean a pair $(G,\omega)$, where $G$ is a locally compact group and $\omega$ is a weight function on $G$. Note that $L^p(G,\omega)$ as a Banach space is isometrically isomorphic to $L^p(G)$ via the mapping $L^p(G,\omega)\to L^p(G);\ f\mapsto f\omega$. 

Throughout the paper, $\chi_E$ denotes the characteristic function of a set $E$. Note that the space $L^p(G,\omega)$ contains $\chi_E$ for every set $E\subset G$ of finite measure.

Let $(G,\omega)$ be a weighted locally compact group. For $x\in G$, define the left and right translation operators by
$$
l_xf(y)=f(x^{-1}y), \qquad r_xf(y)=f(yx^{-1}),
$$
where $f\in L^p(G,\omega)$ and $y\in G$. It follows from our assumptions on the weight that both $l_x$ and $r_x$ are bounded on $L^p(G,\omega)$, with
$$
\|l_x\|\leq\omega(x) \text { and } \|r_x\|\leq\omega(x)\Delta(x)^{-1/p},
$$
where $\Delta$ is the modular function of $G$.

We recall that a linear operator $L$ on $L^p(G,\omega)$ is a left multiplier if
$$
L(f\ast g)=L(f)\ast g, \qquad f,g\in L^p(G,\omega).
$$ 
In particular, the mapping $f\mapsto l_xf$ is a left multiplier on $L^p(G,\omega)$. 

We conclude this section with some lemmas that are needed for our arguments in Sections \ref{3}.

The following lemma is an analogue of \cite[Proposition 1]{MR0193531}, for weighted $L^p$-algebras. The proof is an easy modification of the proof of \cite[Proposition 1]{MR0193531} and is therefore omitted.
\begin{lemma}\label{lemma1}
A bounded operator $L:L^p(G,\omega)\to L^p(G,\omega)$ is a left multiplier if and only if  $Lr_x = r_xL$, for all $x\in G$.
\end{lemma}

The following is \cite[Lemma 4.1]{kuz-mz}:
\begin{lemma}\label{lem}
Let $(G,\omega)$ be a weighted locally compact group and $f\in L^p(G,\omega)$ be given. The mapping $x\mapsto l_xf; G\to L^p(G,\omega)$ is continuous.
\end{lemma}

\begin{lemma}\label{l1}
Let $(G,\omega)$ be a weighted locally compact group. For every shrinking basis $\{U_{\alpha}\}$ of compact neighbourhoods of the identity $e_G$ we have that
$$
\frac{1}{\lambda(U_\alpha)}\|l_x\raise 1.5 pt\hbox{$\chi$}{}_{U_\alpha}\|^p_{p,\omega} \to \omega^p(x).
$$
\end{lemma}
\begin{proof}
We have
$$
\frac{1}{\lambda(U_{\alpha})}\int_{U_\alpha}l_x\omega^p(y)dy-\omega^p(x)=\frac{1}{\lambda(U_{\alpha})}\int_{U_\alpha}\Big( \omega^p(xy)-\omega^p(x) \Big)\ dy\to0,
$$
since $\omega^p$ is continuous.
\end{proof}

Before we proceed with the next lemma let us remark that, using an approximate identity of $L^p(G,\omega)$, we can see that if for some $f\in L^p(G,\omega)$
$$
h\ast f=0\qquad h\in L^p(G,\omega),
$$
then we must have that $f=0$, almost everywhere.

\begin{lemma}\label{lastlemma}
If for some positive number $\gamma$ and $x\in G$, we have that $\gamma \,l_x$ is both a left multiplier of $L^p(G,\omega)$ and an algebra isomorphism, then $\gamma=1$ and $x=e_G$.
\end{lemma}
\begin{proof}
Since $\gamma\, l_x$ is both an algebra isomorphism and a multiplier we have that
$$
\gamma\, l_x(f)\ast g=\gamma l_x(f\ast g)=\gamma\, l_x(f)\ast\gamma\, l_x(g),\qquad f,g\in L^p(G,\omega).
$$
That is,
$$
\gamma \,l_x(f)\ast(g-\gamma\, l_x(g))=0,\qquad f,g\in L^p(G,\omega).
$$
So, $g-\gamma \,l_x(g)=0$, that is $g=\gamma\, l_x(g)$ for every $g\in L^p(G,\omega)$. In particular, for every $g\in C_c(G)$ and $y\in G$, $g(y)=\gamma\, g(xy)$. It now follows that $x=e_G$ and $\gamma=1$.
\end{proof}

\section{ Biseparating algebra isomorphisms between weighted $L^p$-algebras}\label{3}

In this section, we first show that every bicontinuous biseparating algebra isomorphism $T:L^p(G,\omega_1)\to L^p(H,\omega_2)$ is a weighted composition operator. We then use this result to characterize both bipositive algebra isomorphisms ($p\geq1$) and isometric algebra isomorphisms ($p\geq1$, $p\neq2$) between weighted $L^p$-algebras.

 Throughout this section, we assume that $p\ge 1$. An operator $T:L^p(X,\Sigma_1,\mu)\to L^p(Y,\Sigma_2,\nu)$ between measure spaces $(X,\Sigma_1,\mu)$ and $(Y,\Sigma_2,\nu)$ is called disjointness preserving if $f\cdot g=0$  implies that $Tf\cdot Tg=0$. The operator $T$ is called biseparating if $T$ is a bijection and both $T$ and $T^{-1}$ are disjointness preserving.

It is convenient to use the following terminology, as it was done in Definition 2.1 of \cite{MR3589852} in a more restrictive form. Let $(G,\omega_1)$ and $(H,\omega_2)$ be weighted locally compact groups. A standard isomorphism $(\gamma,\phi)$ from $(G,\omega_1)$ onto $(H,\omega_2)$ consists of the following data: an isomorphism of locally compact groups $\phi:G\to H$, and a continuous character $\gamma:G\to \mathbb{C}^+$ such that 
\begin{equation}\label{equ01}
\begin{split}
0<\inf_{x\in G}\;\lvert\gamma(x)\rvert\;\frac{(\omega_2\circ\phi)(x)}{\omega_1(x)}\text{\ \ \ and\ \  \ } \sup_{x\in G}\lvert\gamma(x)\rvert\;\frac{(\omega_2\circ\phi)(x)}{\omega_1(x)}<\infty.
\end{split}
\end{equation}

Let $(G,\omega_1)$ and $(H,\omega_2)$ be weighted locally compact groups, and suppose that $(\gamma,\phi)$ is a standard isomorphism from $(G,\omega_1)$ to $(H,\omega_2)$. It is known that in this case there exists a {\it measure adjustment constant } $c\in\R$ such that
$c=\displaystyle{\frac{\lambda_1(E)}{\lambda_2(\phi(E))}}$ for every non-null measurable set $E\subset G$, where $\lambda_1$ and $\lambda_2$ are Haar measures on $G$ and $H$, respectively. Define $T_{\gamma,\phi}:L^p(G,\omega_1)\to L^p(H,\omega_2)$ by
\begin{equation}\label{T_gamma_phi}
T_{\gamma,\phi}(f)=c\,\gamma\circ\phi^{-1}\cdot f\circ\phi^{-1}.
\end{equation}
It is straightforward to check that $T_{\gamma,\phi}$ is a well defined bicontinuous biseparating algebra isomorphism. In what follows we show that the converse is also true: every bicontinuous biseparating algebra isomorphism is of this form. The main tool for our argument is Proposition \ref{new0} below which is inspired by the Banach-Lamperti theorem stated in \cite[Theorem 3.1]{MR0193531} and also \cite[Theorem 3.2.5]{MR1793753}.

Before we proceed with Proposition \ref{new0}, we need to recall a terminology. Let $(X,\Sigma_1,\mu_1)$ and $(Y,\Sigma_2,\mu_2)$ be measure spaces.  A {\it regular set isomorphism} of $\Sigma_1$ into $\Sigma_2$, defined modulo null sets, is a mapping $\Phi$ satisfying the following conditions:
\begin{enumerate}[(i)]
\item $\Phi(X\setminus A)=\Phi X\setminus \Phi A$ for all $A\in\Sigma_1$;
\item $\Phi(\cup_1^\infty A_n)=\cup_{n=1}^\infty\Phi(A_n)$ for disjoint $A_n\in\Sigma_1$;
\item $\mu_2(\Phi A)=0$ if and only if $\mu_1(A)=0$.\label{this}
\end{enumerate}

 The Banach-Lamperti theorem deals with isometric isomorphisms, but its proof goes by first showing that the map is disjointness preserving, what we impose in Propositions \ref{new0} and \ref{new1} as an assumption. Thus, there is no need to provide a proof of the Proposition.
 
\begin{prop}\label{new0}
Let $(X,\Sigma_1,\mu)$ be a sigma-finite measure space, $(Y,\Sigma_2,\nu)$ be a measure space and $T:L^p(X,\Sigma_1,\mu)\to L^p(Y,\Sigma_2,\nu)$ be a linear disjointness preserving operator. Then there are a regular set isomorphism $\Phi$ and a function $h$ on $Y$ such that for every set $E$ of finite measure
$$
T(\chi_E)=h\cdot\chi_{\Phi(E)}.
$$
\end{prop}

Decomposing the group $G$ into cosets by an open sigma-compact subgroup and applying the Banach--Lamperti theorem to every coset, one proves a variation of Banach-Lamperti theorem which we will use:
\begin{prop}\label{new1}
Suppose that $(G,\omega)$ and $(H,\omega_2)$ are weighted locally compact groups and $T:L^p(G,\omega_1)\to L^p(H,\omega_2)$ is a linear disjointness preserving operator. Then there is a regular set isomorphism $\Phi$ between the sigma-algebras of Haar measurable subsets of $G$ and $H$ and a function $h$ on $H$ such that
$$
T(\chi_E)=h\cdot\chi_{\Phi(E)}
$$
for every set $E\subset G$ of finite measure.
\end{prop}

In the argument of our proof we make use of some particular left multipliers.

\begin{lemma}\label{k}
Every bipositive left multiplier $L$ on $L^1(G,\omega)$ is a positive multiple of a left shift.
\end{lemma}
\begin{proof}
Let $L$ be a bipositive left multiplier on $L^1(G,\omega)$. Then by \cite[Lemma 4.1]{MR3589852}, there is a positive measure $\mu\in M(G,\omega)$ such that $L(f)= \mu\ast f$, $f\in L^1(G,\omega)$. Since $L^{-1}$ is also a positive left multiplier, it follows that $\mu$ is an invertible positive measure with a positive inverse. Therefore by \cite[Proposition 1.3.(i)]{MR3589852}, $\mu$ must be a positive multiple of a pointmass, say $\alpha\delta_x$, and therefore $L(f)=\alpha l_x(f)$, for every $f\in L^1(G,\omega)$.     
\end{proof}

\begin{lemma}\label{k4}
Every bicontinuous biseparating left multiplier $L$ on $L^p(G,\omega)$ is a scalar multiple of a left shift.
\end{lemma}
\begin{proof}
Let $L:L^p(G,\omega)\to L^p(G,\omega)$ be a biseparating left multiplier. We will show that $L$ generates a bipositive left multiplier $L^\prime$ on $L^1(G,\omega)$ such that for a constant $d$ we have that $\frac{1}{d}L=L^\prime$ on characteristic functions of open sets with finite measure. By Lemma \ref{k}, $L^\prime$ is then a scalar multiple of a left shift and therefore $L$ must be a scalar multiple of a left shift.

By Proposition \ref{new1}, there is a measurable function $h$ on $G$ such that for each set of finite measure $E$, we have that $L(\chi_E)=h\cdot \chi_{\Phi(E)}$. We will show that $h$ is constant.

First note that $\Phi(G)=G$ (up to a null set), otherwise one could pick a subset $Y\subset G\setminus\Phi(G)$ with $0<\lambda(Y)<\infty$, and then we would have $\chi_Y$ and $L(\chi_E)$ have disjoint supports for every measurable subset $E$ of finite measure, whence $L^{-1}(\chi_Y)$ and $\chi_E$ have disjoint supports, and thus $L^{-1}(\chi_Y)=0$; but this contradicts the assumption that $L^{-1}$ is a bijection.

Now suppose that $h$ is not constant. Then there exist $x\in G$ and a set $E$ of nonzero finite measure such that $r_xh(y)\neq h(y)$ for $y\in E$. Moreover, we can assume that $\lambda(\Phi^{-1}(E))<\infty$ and $\lambda(\Phi^{-1}(Ex^{-1}))<\infty$. Set $F=\Phi^{-1}(Ex^{-1})\cup\Phi^{-1}(E)x^{-1}$, then $Ex^{-1}\subseteq \Phi(F)$ and $E\subseteq\Phi(Fx)$. For $y\in E$, 
we have 
$$
r_xL(\chi_F)(y)=L\chi_F(yx^{-1})=h(yx^{-1})\chi_{\Phi(F)}(yx^{-1})=h(yx^{-1})
$$
and  
$$
Lr_x\chi_F(y)=L\chi_{Fx}(y)=h(y)\chi_{\Phi(Fx)}=h(y).
$$
Thus, $h(yx^{-1})=h(y)$ for almost every $y\in E$, which is a contradiction. Therefore, there is a constant $d$ such that 
\begin{equation}\label{s1}
L(\chi_{E})=d\chi_{\Phi(E)}
\end{equation}
for every measurable set $E\subset G$ of finite measure. Since $d\lambda(\Phi(E))^{1/p}\le\|L\|\lambda(E)^{1/p}$, it follows easily that $\frac{1}{d}L$ is a bounded map on every $L^p$ class, in particular on $L^1$. From \eqref{s1} it is seen that $\frac{1}{d}L$ is bipositive, and hence by Lemma \ref{k}, $L$ is a multiple of a left shift.
\end{proof}

\begin{theorem}\label{bithm}
Let $(G, \omega_1)$ and $(H, \omega_2)$ be weighted locally compact groups. Let $T:L^p(G,\omega_1)\to L^p(H,\omega_2)$ be a bicontinuous biseparating algebra isomorphism. Then there exists a standard isomorphism $(\gamma,\phi)$ from $(G,\omega_1)$ to $(H,\omega_2)$, such that $T=T_{\gamma,\phi}$.
\end{theorem}
\begin{proof}
Let $T:L^p(G,\omega_1)\to L^p(H,\omega_2)$ be a biseparating algebra isomorphism. For every $x\in G$, the mapping 
$$Tl_xT^{-1}:L^p(H,\omega_2)\to L^p(H,\omega_2)$$
is a biseparating left multiplier. Therefore, by Lemma \ref{k4}, $Tl_xT^{-1}$ is a scalar multiple of a left translation on $L^p(H,\omega_2)$. Let $\phi:G\to H$ and $\gamma:G\to\mathbb{C}$ be such that $$Tl_xT^{-1}=\gamma(x)l_{\phi(x)}.$$

Let $\{U_\alpha\}_{\alpha\in I}$ be a shrinking basis of symmetric precompact neighbourhoods of identity in $H$ all contained in some fixed precompact neighbourhood of identity, say $U$. Since $T$ and $T^{-1}$ are bounded operators and $\|l_x\|\leq\omega_1(x)$, we have, for every $\alpha$
\begin{equation*}
\begin{split}
\|\gamma(x)l_{\phi(x)}(\chi_{U_\alpha})\|_{p,\omega_2}&=\|Tl_xT^{-1}(\chi_{U_\alpha})\|_{p,\omega_2}\leq\|T\|\;\|T^{-1}\|\;\omega_1(x)\;\|\chi_{U_\alpha}\|_{p,\omega_2}.
\end{split}
\end{equation*}
Since by Lemma \ref{l1}, $\displaystyle{\frac{1}{\lambda_2(U_\alpha)}\|l_{\phi(x)}(\chi_{U_\alpha})\|_{p,\omega_2}^p\to\omega_2^p(\phi(x))} \text{ and }\displaystyle{\frac{1}{\lambda_2(U_\alpha)}\|\chi_{U_\alpha}\|_{p,\omega_2}^p\to1},$ we have that
\begin{equation}\label{long}
\begin{split}
\lvert\gamma(x)\rvert\;\omega_2(\phi(x))&=\lim_{\alpha}\frac{\lvert\gamma(x)\rvert\;\| l_{\phi(x)}(\chi_{{U_\alpha}})\|_{p,\omega_2}}{\lambda_2(U_\alpha)^{\frac{1}{p}}}\\
&\le\lim_\alpha\frac{\|T\|\;\|T^{-1}\|\;\omega_1(x)\;\|\chi_{U_\alpha}\|_{p,\omega_2}}{\lambda_2(U_\alpha)^{\frac{1}{p}}}\leq\|T\|\;\|T^{-1}\|\omega_1(x).
\end{split}
\end{equation}
Let now $\{V_\alpha\}_{\alpha\in J}$ be a shrinking basis of symmetric precompact neighbourhoods of identity in $G$.
Since $T^{-1}(Tl_xT^{-1})T(\chi_{V_\alpha})=l_x(\chi_{V_\alpha})$, we have
\begin{equation*}
\begin{split}
\|l_x(V_{\alpha})\|_{p,\omega_1}=\|T^{-1}(Tl_xT^{-1})T(\chi_{V_\alpha})\|_{p,\omega_1}&\leq\|T^{-1}\|\|\gamma(x)\; l_{\phi(x)}T(\chi_{V_\alpha})\|_{p,\omega_2}\\
&\leq\|T^{-1}\|\;\lvert\gamma(x)\rvert\;\omega_2(\phi(x))\;\|T(\chi_{V_\alpha})\|_{p,\omega_2}\\
&\leq\|T^{-1}\|\;\lvert\gamma(x)\rvert\;\omega_2(\phi(x))\;\|T\|\;\|\chi_{V_\alpha}\|_{p,\omega_1}.
\end{split}
\end{equation*}
Using the same argument as in (\ref{long}), we have that
\begin{equation}\label{gamma-omega}
\omega_1(x)\leq\|T^{-1}\|\;\|T\|\;\lvert\gamma(x)\rvert\;\omega_2(\phi(x)).
\end{equation}
Thus, for the mappings $\phi:G\to H$ and $\gamma:G\to\mathbb{C}^*$ we have that 
\begin{equation}\label{equ01}
\begin{split}
0<\inf_{x\in G}\lvert\gamma(x)\rvert\;\frac{(\omega_2\circ\phi)(x)}{\omega_1(x)}\text{\ \ \ and\ \  \ } \sup_{x\in G}\lvert\gamma(x)\rvert\;\frac{(\omega_2\circ\phi)(x)}{\omega_1(x)}<\infty.
\end{split}
\end{equation}
We show next that $\phi:G\to H$ is an isomorphism of locally compact groups and $\gamma:G\to\mathbb{C}^*$ is a continuous character. First we show that $\phi$ and $\gamma$ are homomorphisms. To see this, first note that for every $x$ and $y$ in $G$ we have that
\begin{equation}\label{eq1}
\gamma(xy)l_{\phi(xy)}=Tl_{xy}T^{-1}=Tl_xT^{-1}Tl_yT^{-1}=\gamma(x)\gamma(y)l_{\phi(x)\phi(y)}.
\end{equation}
Let $f\in C_c(H)$ be such that $f(\phi(xy)^{-1})=f(\phi(y)^{-1}\phi(x)^{-1})=1.$ Then using equation (\ref{eq1})
\begin{equation*}
\begin{split}
\gamma(xy)&=\gamma(xy)l_{\phi(xy)}f(e_G)=\gamma(x)\gamma(y)l_{\phi(x)\phi(y)}f(e_G)=\gamma(x)\gamma(y),
\end{split}
\end{equation*}
and therefore $\gamma$ is a homomorphism. It now follows from equation (\ref{eq1}) that $l_{\phi(xy)}=l_{\phi(x)\phi(y)}$. In particular this means that for every function $f\in C_c(H)$ we have that $$f(\phi(xy)^{-1})=l_{\phi(xy)}f(e_G)=l_{\phi(x)\phi(y)}f(e_G)=f(\phi(y)^{-1}\phi(x)^{-1}).$$
Since $G$ is a completely regular topological space, we must have that $\phi(xy)^{-1}=\left(\phi(x)\phi(y)\right)^{-1}$, and therefore, $\phi$ is a homomorphism. 

Now we show that $\phi$ and $\gamma$ are continuous. To see this, suppose that $x$ in $G$ is given, and suppose that $(x_\alpha)$ is a net in $G$ converging to $x$. We first show that $\phi(x_\alpha)\to\phi(x)$. Let $\phi(x)U$ be a compact neighbourhood of $\phi(x)$. Let $W$ be a neighbourhood of the identity such that $WW^{-1}\subseteq U$. Then for every $f\in L^p(H,\omega_2)$, using Lemma \ref{lem} and since both $T$ and $T^{-1}$ are bounded, we have that $Tl_{x_{\alpha}} T^{-1}(f)\xrightarrow{\|.\|_{p,\omega_2}} Tl_xT^{-1}(f)$, that is $\gamma(x_{\alpha})l_{\phi(x_{\alpha})}f\xrightarrow{\|.\|_{p,\omega_2}}\gamma(x)l_{\phi(x)}f$. 
 Setting $f=\chi_W$, we see that for $\epsilon_0:=\|\gamma(x)l_{\phi(x)}(\chi_W)\|_{p,\omega_2}$ there is $\alpha_0$ such that
$$
\|\gamma(x_{\alpha})l_{\phi(x_{\alpha})}(\chi_W)-\gamma(x)l_{\phi(x)}(\chi_W)\|_{p,\omega_2}<\epsilon_0
$$
for every $\alpha>\alpha_0$. We claim that for every $\alpha>\alpha_0$, $\phi(x_{\alpha})$ belongs to $\phi(x)U$. Suppose on the contrary, that for some $\alpha_1>\alpha_0$, $\phi(x_{\alpha_1})\not\in \phi(x)U$. Then, $\phi(x_{\alpha_1})W\cap \phi(x)W=\emptyset$, so the functions $l_{\phi(x_{\alpha_1})}(\chi_W)$ and $l_{\phi(x)}(\chi_W)$ have disjoint supports. Therefore,
\begin{equation*}
\begin{split}
\|\gamma(x_{\alpha_i})l_{\phi(x_{\alpha_i})}(\chi_W)-\gamma(x)l_{\phi(x)}(\chi_W)\|_{p,\omega_2}&=\|\gamma(x_{\alpha_i})l_{\phi(x_{\alpha_i})}(\chi_W)\|_{p,\omega_2}+\|\gamma(x)l_{\phi(x)}(\chi_W)\|_{p,\omega_2}\\
&\geq \|\gamma(x)l_{\phi(x)}(\chi_W)\|_{p,\omega_2}=\epsilon_0,
\end{split}
\end{equation*}
a contradiction. Thus, we must have $\phi(x_{\alpha})\to\phi(x)$.

We show next that $(\gamma(x_\alpha))$ converges to $\gamma(x)$. Suppose that on the contrary $\gamma(x_\alpha)\not\to\gamma(x)$. Then by going through a subnet we can assume that there is $\epsilon_0>0$ such that $\lvert\gamma(x_\alpha)\to\gamma(x)\rvert>\epsilon_0$. Since $\phi(x_\alpha)\to\phi(x)$, there is a compact neighbourhood of identity $W$ such that $\phi(x_\alpha)\in \phi(x)WW^{-1}$ (note that $\phi(x_\alpha)W\cap\phi(x)W$ is a nonempty compact set). Set $m:=\inf\{\omega_2^p(y) : y\in \phi(x_\alpha)^{-1}W\cap\phi(x)^{-1}W \}$, then we have that
\begin{equation*}
\begin{split}
\|\gamma(x_\alpha)l_{\phi(x_\alpha)}(\chi_W)-\gamma(x)l_{\phi(x)}(\chi_W)\|^p_{p,\omega}&=\int\lvert\gamma(x_\alpha)\chi_{\phi(x_\alpha)W}-\gamma(x)\chi_{\phi(x)W}\rvert^p(y)\omega_2^p(y)\, dy\\
&\geq m\int_{\phi(x_\alpha)W\cap\phi(x)W}\lvert\gamma(x_\alpha)-\gamma(x)\rvert^p\ dy\geq m\epsilon_0^p,
\end{split}
\end{equation*}
a contradiction. So,  $(\gamma(x_\alpha))$ converges to $\gamma(x)$.

Now we show that $\phi$ is a homeomorphism. Applying the reasoning above to $T^{-1}$, we conclude that there are continuous homomorphisms $\phi^{\prime}:H\to G$ and $\gamma^{\prime}:H\to \mathbb{C}$ such that $T^{-1}l_yT=\gamma^{\prime}(y) l_{\phi^{\prime}(y)}$ for every $y\in H$. Now, for each $g\in L^p(H,\omega_2)$, if we let $f:=T^{-1}(g)$, we have that
\begin{equation*}
\begin{split}
l_yg=T(T^{-1}l_yTf)&=T(\gamma^{\prime}(y)l_{\phi^{\prime}(y)}f)\\
&=T(\gamma^{\prime}(y)l_{\phi^{\prime}(y)}T^{-1}g)=\gamma^{\prime}(y) \gamma\circ\phi^{\prime}(y)l_{\phi\circ\phi^{\prime}(y)}g.
\end{split}
\end{equation*}
Thus, $\phi(\phi^{\prime}(y))=y$, for every $y$ in $H$. Similarly, we can see that $\phi^{\prime}(\phi(x))=x$, for every $x$ in $G$. Therefore, $\phi$ is a bijection. Since $\phi^{\prime}$ is continuous, $\phi$ is a homeomorphism.

Now, consider the biseparating algebra isomorphism $T_{\gamma,\phi}:L^p(G,\omega_1)\to L^p(H,\omega_2)$. A straightforward calculation shows that for every $y$ in $H$,
$$
TT_{\gamma,\phi}^{-1}r_yT_{\gamma,\phi}T^{-1}=r_y.
$$
Therefore, by Lemma \ref{lemma1}, the biseparating algebra isomorphism $TT_{\gamma,\phi}^{-1}$ is also a left multiplier. It follows from Theorem \ref{k4} that there exist a constant $\zeta$ and an element $z$ in $H$ such that $TT_{\gamma,\phi}^{-1}=\zeta l_z$. Now by Lemma \ref{lastlemma} we have that $TT_{\gamma,\phi}^{-1}$ is the identity operator on $L^p(H,\omega_2)$, and so $T=T_{\gamma,\phi}$.
\end{proof}

We now proceed to bipositive algebra isomorphisms and isometric algebra isomorphisms. An element $f\in L^p(G,\omega)$ is called positive if $f(x)\geq0$, $\lambda$-a.e.\ for $x\in G$. An operator $T:L^p(G,\omega_1)\to L^p(H,\omega_2)$ is positive if $T(f)\geq0$ whenever $f\geq0$. The operator $T$ is called bipositive if $T$ is a bijection and both $T$ and $T^{-1}$ are positive operators. 

\begin{prop}\label{p}
Let $T:L^p(G,\omega_1)\to L^p(H,\omega_2)$ be an algebra isomorphism such that $T$ is bipositive, or $T$ is an isometry and $p\neq2$; then $T$ is bicontinuous and biseparating.
\end{prop}
\begin{proof}
Suppose first that $T$ is bipositive. By \cite[Thm.4.3]{MR809370}, $T$ and $T^{-1}$ are bounded operators. Positivity implies that $T$ maps the subspace of real-valued functions $L^p_\R(G,\omega)\subset L^p(G,\omega)$ to itself, and on this subspace $|T(f)|=T(|f|)$ (see also \cite{Meyer-Nieberg}). For $f,g\in L^p_\R(G,\omega)$ then $f\cdot g=0$ implies $|f\wedge g|=0$ and $|T(f)\wedge T(g)|=0$, so that $T$ is disjointness preserving on $L^p_\R(G,\omega)$. It is now not hard to see that $T$ is disjointness preserving on all of $L^p(G,\omega)$. By symmetry $T^{-1}$ is also disjointness preserving and therefore $T$ is biseparating.

If $T:L^p(G,\omega_1)\to L^p(H,\omega_2)$ is a linear isometry, since Banach spaces $L^p(G,\omega_1)$ and $L^p(H,\omega_2)$ are isometric to $L^p(G)$ and $L^p(H)$, respectively, it follows from \cite[Corollary 2.1]{MR0193531} that $T$ is disjointness preserving. 
\end{proof} 
\begin{corollary}
Let $T:L^p(G,\omega_1)\to L^p(H,\omega_2)$ be an algebra isomorphism.

$(i)$ If $T$ is  bipositive then $T=T_{\gamma,\phi}$ for some standard isomorphism $(\gamma,\phi)$ from $(G,\omega_1)$ to $(H,\omega_2)$ with $\gamma:G\to\mathbb{R^+}$.

$(ii)$ If $T$ is isometric and $p\ne2$ then $T=T_{\gamma,\phi}$ for some standard isomorphism and the ratio $\omega_1/\omega_2\circ\phi$ is a homomorphism.
\end{corollary}
\begin{proof}
$(i)$ It follows from Theorem \ref{bithm} and Proposition \ref{p} that $T=T_{\gamma,\phi}$ for some standard isomorphism $(\gamma,\phi)$ from $(G,\omega_1)$ to $(H,\omega_2)$. Since both $T$ and $T^{-1}$ are bipositive the multiplier $Tl_xT^{-1}$ in the proof of Theorem \ref{bithm} is also bipositive, so that $\gamma(x)>0$.

$(ii)$ It follows from Theorem \ref{bithm} and Proposition \ref{p} that $T=T_{\gamma,\phi}$ for some standard isomorphism $(\gamma,\phi)$ from $(G,\omega_1)$ to $(H,\omega_2)$. Since $\|T\|=\|T^{-1}\|=1$, inequalities \eqref{long} and \eqref{gamma-omega} imply that
$$
\omega_1(x)\leq|\gamma(x)|\;\omega_2(\phi(x))\le \omega_1(x).
$$
Thus, $\displaystyle{|\gamma|=\frac{\omega_1}{\omega_2\circ\phi}}$, so that this ratio is multiplicative.
\end{proof}

\section{Algebra automorphisms of $\ell^p(\Z,\omega)$}\label{4}

In this section we describe automorphisms of a family of weighted algebras $\ell^p(\Z,\omega)$, applying known resutls on composition operators and in particular the Beurling-Helson theorem.

Set $\omega_a(n)=\max(1,|n|^a)$, $a\ge0$. It is known that for $a\ge1$, the weighted space $\ell^p(\Z,\omega_a)$ is a convolution algebra. For $a>b$, we have $\ell^p(\Z,\omega_a)\subset \ell^p(\Z,\omega_b)$.

\let\cal\mathcal
\let\hat\widehat
\begin{theorem}\label{reduction_of_a}
Suppose that $a\in\Z$ and either $a\ge1$ and $p=1$ or $a\ge2$ and $p>1$. Let $T:\ell^p(\Z,\omega_a)\to \ell^p(\Z,\omega_a)$ be an algebra isomorhism. Then $T$ can be extended to a continuous operator on $\ell^p(\Z)$ and this extention is a linear isomorphism.
\end{theorem}
\begin{proof}
Denote the space $\ell^p(\Z,\omega_a)$ by $\cal L_{p,a}$. Note that $\cal L_{p,a}\subset \cal L_{1,a-1}$ (this is checked by a direct calculation). It is known that $\cal L_{p,a}$ is an algebra with respect to convolution if $p=1$, $a\ge0$ and $p>1$, $a\ge1$. Moreover, $\cal L_{p,a}$ is a module over $\cal L_{1,a}$ for every $a\ge0$. The spectrum of every algebra in this class is the unit circle $\T$ \cite{grsh,kuz-msb}. Further we denote by $\hat f$ the Gelfand transform of $f\in\cal L_{p,a}$, $\hat f:\T\to\C$. Consider the formal derivation on $\cal L_{p,a}$: $(f')_n=(n+1)f_{n+1}$. It is readily seen that $f'\in \cal L_{p,a-1}$ for every $f\in\cal L_{p,a}$.

Let us prove first that $T$ can be extended to $\cal L_{p,a-1}$.
For every $z\in\T$, the map $f\mapsto \widehat{(Tf)}(z)$ is a character of $\cal L_{p,a}$, which is obviously nonzero, so that there exists $u_z\in\T$ such that $\widehat{(Tf)}(z)=\hat f(u_z)$. Denote $\phi(z)=u_z$; we get a map $\phi:\T\to\T$.
Denote by $\delta_1\in\cal L_{p,a}$ the indicator function of $1$; one can verify that $\phi=\widehat{T\delta_1}$. This implies, in particular, that $\phi\in\hat{\cal L}_{p,a}$.

Since $T$ is invertible, the equation $Tf=\delta_1$ has a solution in $\cal L_{p,a}$ which has of course the Gelfand transform $\phi^{-1}$;  this implies that $(\phi^{-1})'=1/\phi'\in\hat{\cal L}_{p,a-1}$.

For every $f\in\cal L_{p,a}$, we have by assumption $\widehat{Tf}=\hat f\circ\phi\in\hat{\cal L}_{p,a}$. The formal derivative of a composition satisfies the usual formula (what can be checked by continuity and by its values on $\delta_1^n$); then, $(\hat f\circ\phi)' = (\hat f'\circ\phi)\cdot\phi'\in\hat{\cal L}_{p,a-1}$.
Since $1/\phi'\in\hat{\cal L}_{p,a-1}$ and this is an algebra, we have also $\hat f'\circ\phi\in\hat{\cal L}_{p,a-1}$.

For $g=(g_n)\in\cal L_{p,a-1}$, set $f_0=0$ and $f_n=g_n/n$ for $n\ne0$. We have $f\in\cal L_{p,a}$, and one calculates that $\hat f'(z)=\sum nf_n z^{n-1} = \sum_{n\ne0} g_n z^{n-1}=(\hat g(z)-g_0)/z$. As $\hat f'\circ\phi\in\hat{\cal L}_{p,a-1}$, it follows that $(\hat g\circ\phi)/\phi\in\hat{\cal L}_{p,a-1}$.

Since $\phi\in\hat{\cal L}_{p,a}$, we have also $\phi\in \hat{\cal L}_{1,a-1}$, so that $\hat g\circ\phi\in \hat{\cal L}_{p,a-1}$ for every $g\in\cal L_{p,a-1}$. The operator $g\mapsto (\hat g\circ\phi)\check{\;}$ is an extension of $T$, and it is clear that it is bounded on $\hat{\cal L}_{p,a-1}$ and is an isomorphism.

Applying this reasoning by induction, we get the statement of the theorem for $p=1$ and $a\ge1$.

For $p>1$, we arrive by induction at an isomorphism of $\cal L_{p,1}$ which is of the form $g\mapsto (\hat g\circ\phi)\check{\;}$ with $\phi\in\cal L_{p,2}$, by the initial assumption of the theorem. The above proof cannot be applied directly (with $a=1$) to extend it to $\cal L_{p,0}=\ell^p(\Z)$ since $\cal L_{p,0}$ is not an algebra anymore. However, this fact ($\cal L_{p,a-1}$ being an algebra) is used only to show that $\hat h/\phi'\in \hat{\cal L}_{p,a-1}$ for all $h\in\cal L_{p,a-1}$, with $\hat h=(\hat f\circ\phi)'$. In our assumptions, the inclusion $\hat h/\phi'\in \hat{\cal L}_{p,0}= \hat{\ell^p(\Z)}$ holds for all $h\in \ell^p(\Z)$ since $1/\phi'\in \hat{\cal L}_{p,a-1}\subset \hat{\cal L}_{1,a-2}\subset \hat{\ell^1(\Z)}$. This allows to conclude, finally, that $T$ extends to $\ell^p(\Z)$, also as a bounded isomorphism.
\end{proof}

\begin{corollary}
Suppose that $a$ is an integer and either $a\ge1$ and $p=1$ or $a\ge2$ and $p>1$. Let $T:\ell^p(\Z,\omega_a)\to \ell^p(\Z,\omega_a)$ be an algebra isomorhism. Then there is $\lambda\in\T$ such that either $(Tf)_n=\lambda^n f_n$ or $(Tf)_n=\lambda^n f_{-n}$ for all $f\in\ell^p(\Z,\omega_a)$ and all $n\in\Z$.
\end{corollary}
\begin{proof}
By Proposition \ref{reduction_of_a}, $T$ extends to an isomorphism of $\ell^p$, which is a composition operator $C_\phi$ with $\phi\in \ell^p(\Z,\omega_a)$.
If $p=1$, it follows from the Beurling-Helson theorem that $\phi$ has the form $\phi(z)=z^m\lambda$ for some $m\in\Z$ and $\lambda\in\T$. Applying the same reasoning to $T^{-1}$, we conclude that $\phi^{-1}(z) = (z/\lambda)^{1/m}$ is also a monomial in $z$, what implies that $m=\pm1$, so that $Tf$ is as described in the statement.

If $p>1$ and $a\ge2$, the result follows from a generalisation of Beurling-Helson theorem given by Lebedev and Olevskii \cite{leb-olev-1994}: if $\phi:\T\to\T$ is a $C^1$ function of the argument and the composition operator given by $\phi$ is bounded on $\ell^p(\Z)$, then $\phi$ is a monomial. In our case, $\phi\in \ell^p(\Z,\omega_a)$, what implies $\phi\in C^1$. The rest is proved as in the $\ell^1$ case.
\end{proof}

For $a=1$ and $p>1$ the question is open.

For fast growing weights we can observe a similar behaviour:

\begin{example} Let $w$ be a weight on $\Z$ of one of the two forms: $w_n=e^{n^\gamma}$, $0<\gamma<1$, or $w_n=a^{|n|}(1+n^2)$, $a>1$. Then the elementary Blaschke product $b_r(z)=\dfrac{z-r}{1-rz}$, $r>0$, does not define a bounded composition operator on $\ell^p(\Z,w)$ for any $p\ne2$.
\end{example}
\begin{proof}
{\it Case 1: $w_n=e^{n^\gamma}$, $0<\gamma<1$}. The spectrum of the algebra $\ell^p(\Z,w)$ is still $\T$, as in the unweighted case \cite{grsh,kuz-msb}. We have $\widehat{C_{b_r}\delta_n} = b_r^n$ on $\T$. Set $\alpha_r = (1+r)/(1-r)$ and $k_n=[\alpha n]$. It is known \cite{sz} that $\check b_r(k_n) \ge c_r n^{-1/3}$ as $n\to\infty$, with a constant $c_r>0$. It follows that $\|(b_r^n) \check{\;}\|_{p,w} = \|C_{b_r}\delta_n\|_{p,w} \ge c_r w_{k_n} n^{-1/3} \ge c_r e^{\alpha^{\gamma}n^\gamma-1} n^{-1/3}$. At the same time, $\|\delta_n\|_{p,w} = w_n = e^{n^\gamma}$. Since $e^{(\alpha^{\gamma}-1)n^\gamma}n^{-1/3}\to\infty$, $n\to\infty$, the operator $C_{b_r}$ cannot be bounded.  

{\it Case 2: $w_n=a^{|n|}(1+n^2)$, $a>1$.} One verifies that $b_r\in\ell^p(\Z,w)$ iff $r<1/a$. For these values of $r$ and positive $n$, we reason as above: $\| (b_r^n) \check{\;}\|_{p,w} = \|C_{b_r}\delta_n\|_{p,w} \ge c_r w_{k_n} n^{-1/3} \ge c_r a^{\alpha n-1} n^{-1/3}$. At the same time, $\|\delta_n\|_{p,w} = w_n = a^{n}(1+n^2)$, and we conclude that $C_{b_r}$ cannot be bounded.  
\end{proof}


We conclude this section with a family of examples of non-standard almost isometric isomorphisms on $\ell^2(\Z,\omega)$. It is well known that $L^2$ spaces admit non-standard isometries, and in particular, the algebras $L^2(G)$ of a compact group $G$ \cite{MR01935310}. Introducing a weight on $G$ can make $L^2(G,\omega)$ a convolution algebra, which would still allow for non-standard automorphisms. We find it interesting, however, to look at the question from the following side.

There is a detailed theory of composition operators on Hardy spaces \cite{shapiro} including their weighted variants \cite{gallardo}. It provides among others almost isometric non-standard operators on the spaces $\ell^2(\Z_+,\omega)$. Below, we show that this gives examples of non-standard almost isometric automorphisms of the algebras $\ell^2(\Z,\omega)$.

\begin{proposition}
Let $a>1$ be an integer and $\omega(n)=\max(1, |n|^a)$. For every $\epsilon>0$, there exists an isomorphism $T_\epsilon: \ell^2(\Z,\omega)\to \ell^2(\Z,\omega)$ of distortion less than $1+\epsilon$ which is not given by a weighted automorphism of $\Z$.
\end{proposition}
\begin{proof}
It is known \cite{zorboska} that the weighted Hardy space $H^2(\omega)$ with the same weight $(\omega_n)_{n\ge0}$ is composition invariant with respect to any holomorphic automorphism of the unit disk. In particular, for the function $\phi(z) = \dfrac {z-r}{1-rz}$, $0<r<1$, the norm of composition operator is bounded by $(1+r)^\Lambda/(1-r)^\Lambda$ with a constant $\Lambda$ depending on the weight \cite{gallardo}. We will show that  the operator $C_\phi: f\mapsto (\hat f\circ \phi)\check{\;}$ is well defined on $\ell^2(\Z,\omega)$ and estimate its norm.

For $f\in \ell^2(\Z,\omega)$, set
$$
f_+(z) = \sum_{n=1}^\infty f_n z^n, \quad f_-(z) = \sum_{n=1}^\infty \bar f_{-n} z^n,
$$
then
$$
\hat f(z) = f_0+ f_+(z) + \overline{ f_-(z)},
$$
and
$$
\hat f\circ\phi(z) = f_0 + f_+\circ\phi(z) + \overline{f_-\circ\phi(z)}.
$$
As $f_+,f_-\in H^2(\omega)$, we have that $f_+\circ\phi$, $f_+\circ\phi$ are in $H^2(\omega)$, so that
\begin{align*}
\hat f\circ\phi(z) &= f_0 + (f_+\circ\phi)(0) + \sum_{n>0} (f_+\circ\phi)_n z^n + \overline{(f_-\circ\phi)(0)} + \sum_{n<0} \overline{(f_-\circ\phi)_{-n}}z^{n}
\\&= f(\phi(0)) + \sum_{n>0} (f_+\circ\phi)_n z^n + \sum_{n<0} \overline{(f_-\circ\phi)_{-n}}z^{n}.
\end{align*}
The three summands above are orthogonal in $\ell^2(\Z,\omega)$, thus
\begin{align*}
\|C_\phi(f)\|_{2,\omega}^2 &= |f(\phi(0))|^2 + \|f_+\circ\phi\|_{2,\omega}^2 + \|f_-\circ\phi\|_{2,\omega}^2
\\& \le |f(\phi(0))|^2 + \Big(\frac{1+r}{1-r}\Big)^{2\Lambda} \Big( \|f_+\|_{2,\omega}^2+\|f_-\|_{2,\omega}^2\Big).
\end{align*}
We can estimate
\begin{align*}
|f(\phi(0))| &= |f(-r)| \le |f_0|+|f_+(-r)|+|f_-(-r)| \le |f_0|+ \sum_{n>0} r^n |f_n| + \sum_{n<0} r^n |f_n|
\\& \le |f_0| + r \Big( \sum_{n\ne0} \frac{r^{2(n-1)}}{\omega(n)^2}\Big)^{1/2} \Big(\sum_{n\ne0} |f_n|^2\omega(n)^2\Big)^{1/2}
\\& \le |f_0|+r \,d \,\|f-f_0\|_{2,\omega},
\end{align*}
where $d = \big(\sum_{n\ne0} \omega(n)^{-2}\big)^{1/2}$. Next,
\begin{align*}
\|C_\phi(f)\|_{2,\omega}^2 &\le \big( |f_0|+r \,d \,\|f-f_0\|_{2,\omega} \big)^2 + \Big(\frac{1+r}{1-r}\Big)^{2\Lambda} \Big( \|f_+\|_{2,\omega}^2+\|f_-\|_{2,\omega}^2\Big)
\\& = |f_0|^2 + 2rd|f_0|\cdot \|f-f_0\|_{2,\omega} + r^2 d^2 \|f-f_0\|_{2,\omega}^2 + \Big(\frac{1+r}{1-r}\Big)^{2\Lambda} \|f-f_0\|_{2,\omega}^2
\\& \le |f_0|^2 + 2rd\|f\|_{2,\omega}^2 + \Big( r^2 d^2 + \Big(\frac{1+r}{1-r}\Big)^{2\Lambda} \Big) \|f-f_0\|_{2,\omega}^2
\\& \le \Big( 2rd + r^2 d^2 + \Big(\frac{1+r}{1-r}\Big)^{2\Lambda} \Big) \|f\|_{2,\omega}^2.
\end{align*}
Denoting
$$
K(r) = \Big( 2rd + r^2 d^2 + \Big(\frac{1+r}{1-r}\Big)^{2\Lambda} \Big)^{1/2},
$$
we have $\|C_\phi\|\le K(r)$.

The inverse operator to $C_\phi$ is $C_\psi$ with the function $\psi(z) = \dfrac{z+r}{1+rz}$ and has the norm $\|C_\psi\|\le K(-r)$. Since $K(r)\to1$, $r\to0$, this proves the proposition.
\end{proof}

\nocite{*}

\bibliographystyle{alpha}

\end{document}